\documentclass[11pt,a4paper]{amsart}
\usepackage{amsmath,amsthm,amssymb,latexsym}
\usepackage{graphicx}

\usepackage[pdftex]{color}

\theoremstyle{plain}
\newtheorem{theorem}{Theorem}[section]

\newtheorem{lemma}[theorem]{Lemma}

\theoremstyle{definition}

\newtheorem{remark}[theorem]{Remark}

\usepackage{hyperref}

 \DeclareMathOperator{\supp}{supp}

  \DeclareMathOperator{\lin}{Lin}
 
 \DeclareMathOperator{\co}{conv}

\newcommand{\N}{\mathbb{N}}
\newcommand{\T}{\mathbb{T}}
\newcommand{\D}{\mathbb{D}}

\newcommand{\C}{\mathbb{C}}
\newcommand{\Z}{\mathbb{Z}}

\renewcommand{\a}{\alpha}
\renewcommand{\b}{\beta}
\renewcommand{\l}{\lambda}

\newcommand{\s}{\sigma}

\newcommand{\p}{\varphi}

\renewcommand{\d}{\delta}
\newcommand{\dd}{\Delta}

\newcommand{\g}{\gamma}
\renewcommand{\gg}{\Gamma}

\newcommand{\z}{\zeta}

\newcommand{\ovl}{\overline}

\newcommand{\mycolon}{{:}\allowbreak\ }

\newcommand{\eps}{\varepsilon}

\renewcommand{\le}{\leqslant}
\renewcommand{\ge}{\geqslant}

\begin{document}

\title[Modulus support functionals]{Modulus support functionals, Rajchman measures and peak functions}

\author[L.~Golinskii]{L. Golinskii}

\address{B. Verkin Institute for Low Temperature Physics and Engineering of the National Academy of Sciences of Ukraine,
47 Nauky ave., Kharkiv, 61103, Ukraine
\newline
\href{https://orcid.org/0000-0002-7677-1210}{ORCID: \texttt{0000-0002-7677-1210} }
}
\email{golinskii@ilt.kharkov.ua}

\author[V. Kadets]{V. Kadets}

\address{V.N.Karazin Kharkiv National University,
4 Svobody sq., Kharkiv, 61022, Ukraine \newline
\href{http://orcid.org/0000-0002-5606-2679}{ORCID: \texttt{0000-0002-5606-2679} }}
\email{v.kateds@karazin.ua}



\thanks{The research of the second author was done in the framework of the Ukrainian Ministry of  Education and Science Research Program 0118U002036 and was partially supported by the project PGC2018-093794-B-I00 (MCIU/AEI/FEDER, UE)}

\begin{abstract}
In 2000 V. Lomonosov suggested a counterexample to the complex version of the Bishop-Phelps theorem on modulus support functionals. 
We discuss the $c_0$-analog of that example and demonstrate that the set of sup-attaining functionals is non-trivial, thus answering 
an open question, asked in \cite{KLMW}.
\end{abstract}

\date{\today}

\maketitle
\thispagestyle{empty}


\section{Introduction}

In the text below, the letter $X$ is used for a Banach space,  $X^*$ is the corresponding dual space, 
$$ \mathbb B(X)=\{x\in X\mycolon \|x\|\le 1\}, \qquad \mathbb S(X)=\{x\in X\mycolon \|x\|=1\} $$ 
stand for its unit ball and sphere, respectively, $M^{\bot}$ is the annihilator in $X^*$ of a closed subspace $M$ of $X$.  
Abbreviation bcc-set means non-empty bounded closed convex set. For a given bcc subset $C \subset X$, a non-zero functional $h \in X^*$
is said to be a \emph{modulus support functional for} $C$ if there is a point $y \in C$ (called a corresponding \emph{modulus support point 
of} $C$) such that 
$$
|h(y)| = \sup_{x \in C} |h(x)|.
$$

We denote by $\D = \{z \in \C \mycolon |z| < 1\}$ the open unit disk in the field of complex numbers, $\T = \{z \in \C \mycolon |z| = 1\}$ 
the unit circle, and $\overline \D = \D \cup \T$ the closed unit disk. 
 
The classical result of Bishop and Phelps \cite{BP1, BP2} says that in every real Banach space $X$ for every bcc subset $C \subset X$ 
the set of modulus support functionals for $C$ is dense in $X^*$. The same question \cite{Phelps1991} for complex linear functionals on a complex Banach space remained open until 2000, when Victor Lomonosov \cite{Lom-2000, Lom-2000+} constructed his ingenious  counterexample in the predual space of $H^\infty$ (see also \cite{Lom-2001} and the next section).

In  \cite{Lom-2001} Lomonosov introduced the following definition: a complex Banach space $X$ has the attainable approximation property (AAP) 
if for any bcc-subset $W \subset X$ the corresponding set of modulus support functionals  is norm dense in $X^*$. 
By weak compactness argument, all reflexive spaces,  in particular $L^p[0,1]$
with $1<p<\infty$ and $\ell^p$ with $1< p<\infty$, enjoy the AAP.   In contrast, $L^\infty[0,1]$ and $\ell^\infty$ contain isometric copies of 
every separable Banach space (see \cite[Section 17.2.4, Exercises 5--8]{Kad2018} or \cite[Theorem 2.5.7]{AlKal}),
so Lomonosov's example can be transferred to those spaces. Consequently, the complex spaces $L^\infty[0,1]$ and $\ell^\infty$ do not have the AAP. 
Surprisingly, for such classical complex spaces as $c_0$ and $L^1[0,1]$, it is unknown whether they possess the AAP or not.

A natural approach to settle the problem in the negative is to transfer somehow the original Lomonosov's example to other spaces and to check if it preserves its properties in this new setting. Such a version of Lomonosov's example for the case of $c_0$ was introduced in the last section of \cite{KLMW}, where it was asked, in particular, if there are any modulus support functionals for that version. 

Although the $c_0$-version of Lomonosov's example and the corresponding question about norm-attaining functionals  were published ``officially'' in 2019, 
they are much older. The example was invented by V.~Kadets in 2003. Since then, the corresponding problem was reported to many colleagues that work in 
Banach space theory or Complex analysis (in particular, to V. Lomonosov), but with no progress. It was a lucky coincidence, that on October 29, 2019, 
the first author of this paper was attending the Kharkiv mathematical society meeting where the second author was advertising this open problem. 

In this note we demonstrate the existence of ``many'' modulus support functionals for the $c_0$-version of Lomonosov's set.
Nevertheless, the more involved question whether the set of modulus support functionals is dense in $c_0^*$ remains open, and so 
the possibility to disprove the AAP for $c_0$ by means of Lomonosov's example is still in doubt.

The structure of the paper is the following. In Section \ref{sect:Lomonosov-original} we recall the basic features of the original Lomonosov's example. 
In Section \ref{sect:Lomonosov-c0} we begin with the precise definition of  $c_0$-version  $S_0\subset \mathbb B(c_0)$ of Lomonosov's set.  
The key message of our note is a tight relation of modulus support functionals for $S_0$ to two notions in harmonic analysis, peak sets and Rajchman measures. 
We reveal this relation later in Section \ref{sect:Lomonosov-c0}, and so reduce the problem of existence of modulus support functionals for $S_0$ to a subtle problem 
of existence of certain Rajchman measures (in one direction the reduction was performed in \cite{KLMW}). 
In the last section we construct such measures as the Cantor measures of constant ratio, and demonstrate the way of generating
infinite families of such measures.

Starting from this point, we deal only with complex Banach spaces.

\section{The original Lomonosov's example} \label{sect:Lomonosov-original}

Let us equip the unit circle $\T$ with the normalized  Lebesgue measure $m(dt)$, and consider the corresponding space $L^1(\T)$. In the standard coupling 
\begin{equation} \label{eq:inner product}
\langle g, x \rangle = \int_{\T} x(t)g(t)\,m(dt), \qquad x\in L^1(\T), \ \ g\in L^\infty(\T), 
\end{equation}
the dual space to $L^1(\T)$ is identified with $L^\infty(\T)$. Let $H^1$ be the standard Hardy space, $H^1_0=t H^1$ be the closed linear span in 
$L^1(\T)$ of the functions $\{t^k\}_{k\ge1}$. Consider the quotient space $X =  L^1(\T)/H^1_0$. Then $X^* = (H^1_0)^\bot$ (see, e.g., 
\cite[Section 9.4.2]{Kad2018}), and so \cite[Chapter VII.A.1]{Koo-80}, 
$$
X^* =\{g \in L^\infty(\T) \colon  \int_{\T}  g(t) t^n m(dt) =0, \quad n = 1, 2, \ldots\} = H^{\infty}.
$$
The space $H^{\infty}$ is known to consist of those functions $g \in L^\infty(\T)$ that can be extended to bounded analytic functions 
in the open unit disk $\D$ in such a way that $\lim_{r \to 1}g(r \zeta) = g(\zeta)$ for almost all $\zeta\in \T$ \cite[Chapter I.D]{Koo-80}. Also, 
$\|g\|_{L^\infty}=\|g\|_{H^\infty}=\sup_{z\in \D}|g(z)|$. $H^{\infty}$ is a unital Banach algebra with the standard product and with 
the identity function ${\bf 1}$ being the unit element.

Each function $g\in H^\infty$ admits the following Cauchy representation (see, e.g., \cite[Chapter II.B.3]{Koo-80}):
\begin{equation}\label{cauchy}
g(z)=\int_{\T} \frac{g(t)}{1-\bar t z}\,m(dt), \quad z\in\D.
\end{equation}
Consider the family of functions
\begin{equation*}
u_z(t):=\frac1{1-\bar t z}=\sum_{k\ge0} z^kt^{-k}\in L^1(\T),
\end{equation*}
and their classes $[u_z]$ in the quotient space $X$. It is clear from \eqref{cauchy}, that each functional $g\in H^\infty$ acts on $[u_z]$ as the
evaluation functional
\begin{equation*}
\langle g, [u_z] \rangle=g(z), \qquad \|[u_z]\|_X=\sup_{\|g\|_\infty=1} |\langle g, [u_z] \rangle|=1.
\end{equation*}

Clearly, $X$ equals the closed linear span of $[u_z]$, $z \in \D$. Denote by $S$ the closed convex hull in $X$ of all $[u_z]$, $z \in \D$. 
This bcc set $S \subset \mathbb B\,(X)$ is a key ingredient of Lomonosov's example mentioned above. The main result of \cite{Lom-2000} says that the only 
modulus support functionals for $S$ are constant functions in $H^\infty$.

Let us briefly recall the Lomonosov's reasoning about modulus support functionals for $S$. First, \cite[Lemma 1]{Lom-2000} states that 
$$
\lim_{k \to \infty} \langle g^k, x \rangle  = 0
$$
for every non-constant $g \in \mathbb B\,(H^{\infty})$ and every $x \in S$. For $x=[u_z]$ this is true since $|g(z)|<1$, $z\in\D$. The rest follows from
the boundedness of the sequence $(g^k)$ together with the pointwise  convergence criterion for functionals (see \cite[Section 17.2.1]{Kad2018}). 
Next,  \cite[Lemma 2]{Lom-2000} says that if $h=h(S)\in H^{\infty}$ is a modulus support functional for $S$, and
$y=y(S) \in S$ is the corresponding modulus support point, then  
$$
\lim_{k \to \infty} \langle h^k, y \rangle  = 1.
$$
This follows from very clever Banach algebra argument: $H^{\infty}$ is a subalgebra of the algebra $C(M)$ of continuous functions on the 
corresponding Gelfand compact, action of $y$ on elements of $H^\infty$ can be represented as integral over $M$ with some Borel 
probability measure $\nu$, and $h(S)$ happens to be identical one on the support of $\nu$. These two results together imply that the only 
possible modulus support functionals for $S$ are constant functions.

\begin{remark} \label{remark21}
There is one more trick from \cite{Lom-2000} that ``kills'' the constant functions: consider instead of $X$ the quotient space 
$X_1 = X/\lin{\delta_0}$, then $X_1^*$ is the subspace of those $g \in H^{\infty}$ that $g(0) = 0$. Then the image $q(S)$ of $S$ 
under the quotient map $q\colon X \to X/\lin{\delta_0}$ is a bcc set in $X_1$ that possess no modulus support functionals at all!
\end{remark}

\section{Modulus support functionals for the space $c_0$} \label{sect:Lomonosov-c0}

In this section we consider the Banach spaces $c_0$ and $c_0^* = \ell^1=\ell^1(\Z_+)$ in the coupling
$$
\langle a, x \rangle = \sum_{n\ge0} x_n a_n, \qquad x\in c_0, \quad a\in\ell^1,
$$
where $x_n$, $a_n$ are the coordinates of vectors $x$ and $a$, respectively. We identify each element $x = (x_j)_{j\ge0} \in c_0$ with the 
function $f_x$ in the unit disk by the rule $f_x(z) = \sum_{n=0}^\infty x_n z^n$ for all $z\in \D$. In this way we identify $c_0$ with the 
corresponding Banach space $c_0(\D)$ of analytic functions having convergent to zero sequences of Taylor coefficients at the origin, equipped 
with the norm $\|f_x\|_{c_0} = \|x\|_\infty =  \max_{n \in \N} |x_n|$. Similarly, we identify $c_0^* = \ell^1$ with the Wiener algebra $W^+$
of analytic in the unit disk functions having absolutely convergent Taylor series
\begin{equation}\label{ident} 
a=(a_j)_{j\ge0}\in\ell^1 \Leftrightarrow f_a=\sum_{n\ge0} a_nz^n\in W^+, \quad \|f_a\|_+=\|a\|_1=\sum_{n\ge0} |a_n|. 
\end{equation}

The functional $a$ is said to be non-trivial if $a \notin \{ (\alpha, 0, 0, \ldots), \alpha \in \C\}$, that is, the function $f_a$ is non-constant. 
We define duality
$$
 \langle f_a, f_x \rangle = \langle a, x \rangle, 
 $$
which agrees in a sense with  the duality formula \eqref{eq:inner product} from Section \ref{sect:Lomonosov-original}.

The set $S_0$, a counterpart of Lomonosov's set above, looks as follows. Given $\l\in\D$, let
$$ \p_\l:=(\l^j)_{j\ge0}\in c_0, \qquad \|\p_\l\|_{c_0}=1, $$
and denote by $S_0$ the closed convex hull in $c_0$ of all $\p_\l$, $\l \in \D$. It is clear that
\begin{equation}\label{eq1} 
\langle a, \p_\l\rangle = f_a(\l), \qquad \forall a\in\ell^1.
\end{equation}
 
To have a new insight on the problem, we recall two notions from the harmonic analysis.

Given a finite complex Borel measure $\mu$ on the unit circle $\T$, its Fourier--Stieltjes coefficients $\widehat\mu(k)$
are defined by the formula 
 $$
\widehat\mu(k) =  \int_{\T} t^{-k} \mu(dt), \qquad k \in \Z.
 $$
The measure $\mu$  belongs to the {\it class} $\mathcal{R}$ (after A. Rajchman), if its 
Fourier--Stieltjes coefficients tend to zero on the left
\begin{equation*} 
\lim_{n\to +\infty} \widehat\mu(-n)=0.
\end{equation*}
As a matter of fact, $\lim_{|n|\to +\infty} \widehat\mu(n)=0$ holds in this case, see \cite[p. 203]{Hel10}.

A closed set $E\subset\T$ of measure zero is said to be a \emph{weak peak set} for $W^+$, if there is a function $g_E\in W^+$, called a \emph{weak peak function}, 
and a complex number $\b\not=0$ so that $g_E=\b$ on $E$ and $\|g_E\|_\infty=|\b|$. Obviously, a closed subset of a weak peak set for
$W^+$ is again a weak peak set.

We will define peak sets and peak functions later in the next section.

The idea of the result below is borrowed from \cite{KLMW}. We present it here for the sake of completeness.

\begin{theorem} \label{thm-example-direct}
Let $b$ be a non-trivial modulus support functional for the set $S_0$. Then there is a measure $\mu\in\mathcal{R}$ such that the sequence of its 
Fourier--Stieltjes coefficients $(\widehat\mu(-n))_{n\ge0}$ is the corresponding to $b$ modulus support point in $S_0$.
Moreover, the set $E=\supp\mu$ is a weak peak set for $W^+$, with $f_b$ being the corresponding weak peak function. 
\end{theorem}
\begin{proof}
The following equality is important in the rest of the paper
\begin{equation} \label{eq2}
\sup_{x\in S_0} |\langle a,x \rangle|=\|f_a\|_\infty, \qquad \forall a\in\ell^1.
\end{equation}
Indeed, for $x\in \co(\p_\l)_{\l\in\D}$, that is,
$$ x=\sum_{k=1}^n w_k\p_{\l_k}, \qquad w_k\ge0, \quad \sum_{k=1}^n w_k=1, $$
we have, by \eqref{eq1},
$$ \langle a,x \rangle=\sum_{k=1}^n w_k \langle a,\p_{\l_k} \rangle=\sum_{k=1}^n w_k f_a(\l_k), \quad |\langle a,x \rangle|\le\|f_a\|_\infty. $$
On the other hand, if $\|f_a\|_\infty=|f_a(t)|$, $t\in\T$, then
$$ \lim_{r\to 1-0}|\langle a,\p_{rt} \rangle| = |f_a(t)|=\|f_a\|_\infty,  $$
as claimed.

Next, let $x=(x_j)_{j\ge0}\in S_0$. Take a convex combination tending to $x$,
\begin{equation}\label{eq2.1} 
w^{(n)}=\bigl(w^{(n)}_j\bigr)_{j\ge0}=\sum_k w_{n,k}\,\p_{\l_k} \rightarrow x, \quad n\to\infty, 
\end{equation}
each sum is finite. In the space $\mathcal{M}(\ovl\D) = C((\ovl\D)^*$ of finite Borel measures on $\ovl\D$ consider the sequence of 
probability measures (of unit total mass)
$$ \mu^{(n)}:=\sum_k w_{n,k}\,\d(\l_k), \qquad \sum_k w_{n,k}=1, $$
where $\d(\l)$ is the Dirac measure at the point $\l\in\ovl\D$. Due to *-weak compactness  of the subset of all probability measures in $\mathcal{M}(\ovl\D)$, we can assume 
(passing to a subsequence, if necessary), that $*-\lim_{n\to\infty}\mu^{(n)}=\mu$ for some probability measure $\mu=\mu_x$. 
For each fixed $j=0,1,2,\ldots$, the latter relation and \eqref{eq2.1} imply
\begin{equation*}
\begin{split}
\int_{\ovl{\D}} \l^j \mu^{(n)}(d\l) &=\sum_{k=1}^n w_{n,k}\,\l_k^j=w^{(n)}_j,  \\
\lim_{n\to\infty}\int_{\ovl{\D}} \l^j \mu^{(n)}(d\l) &=\int_{\ovl{\D}} \l^j \mu_x(d\l)=\widehat\mu_x(-j)=x_j, \quad j=0,1,2,\ldots.
\end{split}
\end{equation*} 
Since $x\in c_0$, we have
\begin{equation}\label{eq3}
\lim_{j\to\infty}\widehat\mu_x(-j)=0,
\end{equation}
so for each $x$ the measure $\mu_x$ belongs to the class $\mathcal{R}$.

Similarly, in view of \eqref{eq1} and \eqref{eq2.1}, for each $a\in\ell^1$ and $x\in S_0$
\begin{equation*}\label{eq3.1}
\begin{split}
\int_{\ovl{\D}} f_a(\l) \mu^{(n)}(d\l) &=\sum_{k=1}^n w_{n,k} f_a(\l_k)=\langle a, w^{(n)}\rangle,  \\
\lim_{n\to\infty}\int_{\ovl{\D}} f_a(\l) \mu^{(n)}(d\l) &=\int_{\ovl{\D}} f_a(\l) \mu_x(d\l)=\langle a, x\rangle.
\end{split}
\end{equation*}

In the case when $a=b$ is a modulus support functional for $S_0$, 
$y \in S_0$ is the corresponding modulus support point, and in view of the definition of modulus support functionals and \eqref{eq2}, 
we come to the main equality  
\begin{equation}\label{eq4}
\left| \int_{\ovl{\D}} f_b(\l) \mu_y(d\l) \right|=|\langle b, y\rangle|=\sup_{x\in S_0} |\langle b, x \rangle|=\|f_b\|_\infty.
\end{equation}
If, in addition, $b$ is non-trivial, it follows from \eqref{eq4} that $\supp\mu_y\subset\T$, and there is a constant $\b\in\C\backslash\{0\}$ such that
\begin{equation}\label{eq5}
f_b(t)=\b, \quad t\in \supp\mu_y; \qquad \|f_b\|_\infty=|\b|.
\end{equation}
So, $\mu=\mu_y$ is a desired measure. The proof is complete.
\end{proof}

\smallskip

Our next goal is to demonstrate that, conversely, each weak peak set $E$ and a measure $\nu\in\mathcal{R}$ with $\supp\nu \subset E$ generate modulus 
support point and functional for $S_0$.
We start with a lemma analogous to the fact that Riemann integral sums of a continuous function approximate the corresponding integral.
  
\begin{lemma}  \label{lem-int-sums}
For each $n\in\N$  we divide $\T$ in $n$  disjoint arcs $\Delta_{n,k}$, $k = 1,2, \ldots, n$ , of equal length:
$$ \dd_{n,k}:=\Bigl[e^{\frac{2(k-1)\pi i}{n}}, e^{\frac{2k\pi i}{n}}\Bigl), \qquad k=1,2,\ldots,n, \quad m(\dd_{n,k})=\frac1n. $$
Put $\z_{n,k}:=r_n\,\exp\bigl(\frac{(2k-1)\pi i}{n}\bigr)\in\D$, where $0<r_n<1$ is taken in such a way that
$$|\z_{n,k}-t|<\frac{\pi}{n}, \quad \forall t\in\dd_{n,k}, \quad k=1,2,\ldots,n. $$ 
Given an arbitrary Borel probability measure $\nu$ on $\T$, denote 
$$ \nu_n:=\sum_{k=1}^n \nu(\dd_{n,k})\,\d(\z_{n,k})\in\mathcal{M}(\ovl\D). $$ 
Then for every continuous function $f$ on $\ovl{\D}$ 
$$\lim_{n \to \infty} \int_{\ovl{\D}} f(\l) \nu_n(d\l)=\int_{\T} f(\l) \nu(d\l).$$
\end{lemma}
 \begin{proof} 
The uniform continuity of $f$ implies that, for each $\eps > 0$, there is $N\in~\N$ such that for every $\z, \tau \in \ovl{\D}$ with 
$|\z - \tau| < \pi N^{-1}$, the inequality $|f(\z) - f(\tau)| < \eps$ holds true. Then, for every $n > N$ we have that
\begin{align*}
  & \left|\int_{\ovl{\D}} f(\l) \nu_n(d\l) - \int_{\T} f(\l) \nu(d\l)\right|  =  
  \left|\sum_{k=1}^n \nu(\dd_{n,k})\,f(\z_{n,k})-\sum_{k=1}^n\int_{\dd_{n,k}}f( t) \nu(dt)\right| \\ 
  &=  \left|\sum_{k=1}^n \left(\int_{\dd_{n,k}}f(\z_{n,k}) \nu(dt) -\int_{\dd_{n,k}}f( t) \nu(dt)\right)\right|  \\ 
  & \le \sum_{k=1}^n \int_{\dd_{n,k}}|f(\z_{n,k}) - f(t)|\,\nu(dt)  < \eps \nu(\T) = \eps.
\end{align*}
\end{proof}

\begin{theorem}  \label{thm-example-inverse}
Let $E$ and $h_E=f_b$, $b\in\ell^1$, be a weak peak set and a corresponding weak peak function for $W^+$, respectively. Let $\nu\in\mathcal{R}$ 
be a probability measure with $\supp\nu\subset E$. Then $b$ is the modulus support functional for $S_0$, and $y=(y_j)_{j\ge0}$, $y_j=\widehat\nu(-j)$, 
is the corresponding modulus support point.
\end{theorem}
\begin{proof}
We show first that $y=(y_j)_{j\ge0}$, $y_j=\widehat\nu(-j)$, belongs to $S_0$.  To this end, note that
$$ \langle a,y\rangle=\sum_{j=0}^\infty a_j\widehat\nu(-j)=\int_{\ovl{\D}} f_a(\l)\nu(d\l), \qquad \forall a\in\ell^1. $$
For each $n\in\N$, consider the arcs $\Delta_{n,k}$, $k = 1,2, \ldots, n$, the  points $\z_{n,k} \in \D$, and the measures $ \nu_n$ 
from Lemma \ref{lem-int-sums}. Denote
\begin{equation*}
\begin{split}
v^{(n)}_j &= \sum_{k=1}^n \nu(\dd_{n,k})\,\z_{n,k}^j=\int_{\ovl{\D}} \l^j\,\nu_n(d\l), \quad j=0,1,\ldots, \\
v^{(n)} &=\bigl(v^{(n)}_j\bigr)_{j\ge0}:=\sum_{k=1}^n \nu(\dd_{n,k})\, \p_{\z_{n,k}} \in S_0. 
\end{split}
\end{equation*}
Then, for each fixed $j=0,1,\ldots$, Lemma \ref{lem-int-sums} with $f(t) = t^j$ gives
\begin{equation*}
\lim_{n\to\infty}|v^{(n)}_j-y_j| = \lim_{n\to\infty} \left|\int_{\ovl{\D}} \l^j  \nu_n(d\l) - \int_\T t^j\nu(dt)\right| = 0.
\end{equation*}
By the weak convergence criterion in $c_0$ (coordinate-wise convergence plus boundedness, see \cite[Section 17.2.3, Theorem 1]{Kad2018}), 
this means that $v^{(n)}$ converge weakly to $y$, so $y$ belongs to the weak closure of $S_0$. 
But a closed convex set is also weakly closed \cite[Section 17.2.3, Theorem 3]{Kad2018}, so $y\in S_0$.

Next, $f_b$ is a weak peak function, and $\supp\nu\subset E$, so   
$$\langle b, y\rangle = \int_{\T} f_b(\l) \nu(d\l)=\int_{E} f_b(\l) \nu(d\l) =\b. $$ 
On the other hand, by \eqref{eq2},
$$ |\langle b, x\rangle|\le \|f_b\|_\infty=|\b|=|\langle b, y\rangle|, \qquad \forall x\in S_0, $$
as stated. The proof is complete. 
\end{proof}

\section{Can a peak set for $W^+$ bear a Rajchman measure?} \label{sect:peak-Rajchman}

The existence of singular measures in the class $\mathcal{R}$ is not obvious at all. It seems that these properties contradict each other, 
and they can hardly be reconciled. Indeed, the support of such measure $\mu$ is ``small'', and, according to the Uncertainty Principle \cite{HavJor-94}, 
this is an obstacle for the spectral smallness of $\mu$ which is now expressed by the Rajchman condition $\lim_{|n|\to\infty} \widehat\mu(n)=0$. 
Nevertheless, the properties are compatible, and we show the examples of such measures. Moreover, the support of the constructed measure will 
be a subset of a peak set for $W^+$.

A closed set $E\subset\T$ of measure zero is said to be a {\it peak set for} $W^+$, if there is a function $g_E\in W^+$, called a {\it peak function} so that
\begin{equation}\label{peak}
g_E(z)=1, \quad z\in E; \qquad |g_E(z)|<1, \quad z\in\ovl{\D}\backslash E.
\end{equation}
It is clear, that each peak set for $W^+$ is a peak set in the weak sense. Conversely, each weak peak set $F$ is a subset of a certain peak set. Indeed, 
let $f_F$ be a corresponding weak peak function so that $f_F=1$ on $F$. Define $E:=\{t\in\T: f_F=1\}\supset F$. It is easy to see that
$$ g_E(z):=\frac{f_F(z)+1}2 $$
is the peak function for the peak set $E$. It is not known, whether each closed subset of a peak set for $W^+$ is again a peak set (this is true for
some other classes of function, such as $A^\a$ below).

Recall the construction of the Cantor set of constant ratio $\xi$, $0<\xi<\frac12$. We start out from the unit interval $[0,1]$ and
remove a concentric open interval (with the center at $1/2$) of the length $1-2\xi$ at the first step. We remove then two concentric open 
intervals of the {\it relative} length $1-2\xi$ from each of two remained closed intervals at the second step, etc. So, at $n$-th step we remove
$2^{n-1}$ concentric open intervals of the relative length $1-2\xi$ from each remained closed interval. Denote by $E_n(\xi)$ the disjoint union
of $2^n$ closed intervals remaining after $n$-th step. The length of each equals $\xi^n$, so $m(E_n(\xi))=(2\xi)^n$, and $E_{n+1}(\xi)\subset E_n(\xi)$.
By the definition, the {\it Cantor set of constant ratio} $\xi$ is
$$ E(\xi)=\bigcap_{n=1}^\infty E_n(\xi), \qquad m(E(\xi))=0. $$
$E$ is a perfect subset of $[0,1]$. The Cantor triadic set arises for $\xi=\frac13$.

To define a related measure, denote by $\s_n(\xi)$ the normalized restriction of the Lebesgue measure on $E_n$. As is known \cite[p. 58]{HavJor-94}, 
the *-weak limit
$$ *-\lim_{n\to\infty} \s_n(\xi)=\s(\xi) $$
exists. It is usually referred to as the {\it Cantor measure of ratio} $\xi$. The measure $\s(\xi)$ is singular continuous, and $\supp \s(\xi)=E(\xi)$.

Any measure on $[0,1]$ can be carried over to a measure on $\T$ in a natural way by means of the mapping $t\to e^{2\pi it}$. We use the same symbol
$\s(\xi)$ for the Cantor measure of ratio $\xi$ on $\T$. An amazing feature of this measure is the fact that its Fourier--Stieltjes coefficients
are available explicitly \cite[Chapter V, (3.5)]{Zyg-02}
$$ \widehat\s_n(\xi)=(-1)^n\,\prod_{k=1}^\infty \cos\bigl(\pi n\xi^{k-1}(1-\xi)\bigr). $$
A complete description of the Cantor measures within the Rajchman class is due to R. Salem: $\s(\xi)\notin\mathcal{R}$ if and only if $\xi^{-1}$ is a
Pisot number, that is, an integer algebraic number with all its conjugates inside the unit disk \cite[Theorem XII.11.8]{Zyg-02}. For rational 
$\xi$ the result was proved earlier by N.K.~Bari, who showed that $\s(\xi)\notin\mathcal{R}$ if and only if $\xi^{-1}$ is a positive integer 
(so the standard Cantor triadic measure is not in $\mathcal{R}$). In conclusion, all Cantor measures $\s(\xi)$ but countably many belong to $\mathcal{R}$.

\smallskip

Going back to peak sets and functions for the Wiener algebra $W^+$, note that, to the best of our knowledge, the subject has not attracted much
attention so far. In contrast, there is a detailed account of such sets and functions for the space $A^\a$ \cite{NoWol-89, No}. By $A^\a$, 
$0<\a\le1$, we denote the class of analytic in $\D$ functions $f$ satisfying a Lipschitz condition of order $\a$
$$ |f(z)-f(w)|\le C|z-w|^\a, \qquad z,w \in\ovl\D. $$
In particular, \cite[Theorem 3.1]{NoWol-89} provides a metric condition on $E$ (in terms of the lengths of the complementary arcs) to be a peak set for $A^\a$, $0<\a<1$.
Precisely, let
$$ \T\backslash E=\bigcup_{n\ge1}\gg_n, $$
a disjoint union of open arcs. Then $E$ is a peak set for $A^\a$ as soon as
\begin{equation}\label{metrcond}
\sum_{n\ge1} m^{1-\a}(\gg_n)<\infty,
\end{equation}
or, equivalently, $d_E^{-\a}\in L^1(\T)$, $d_E(\z)$ is the distance from $\z$ to $E$. Condition \eqref{metrcond} can be easily verified for 
the Cantor sets $E(\xi)$ for certain values of $\xi$. Indeed, we have
$$ m(\gg_j)=\xi^k(1-2\xi), \qquad j=2^k,2^k+1, \ldots,2^{k+1}-1, \quad k=0,1,2,\ldots, $$
and so
$$ \sum_{n\ge1} m^{1-\a}(\gg_n)=\sum_{k=0}^\infty \sum_{j=2^k}^{2^{k+1}-1} \xi^{k(1-\a)}(1-2\xi)^{1-\a}=
(1-2\xi)^{1-\a}\,\sum_{k=0}^\infty (2\xi^{1-\a})^k<\infty, $$
as soon as
\begin{equation}\label{peakcond}
0<\xi<\Bigl(\frac12\Bigr)^{\frac1{1-\a}}.
\end{equation}
We come to the following conclusion: for all but countably many $\xi$ that satisfy \eqref{peakcond}, the Cantor set $E(\xi)$ is the peak set for $A^{\a}$, and
the Cantor measure $\s(\xi)\in\mathcal{R}$, simultaneously.

To complete the argument, we invoke a theorem of S.N.~Bernstein \cite[Theorem VI.3.1]{Zyg-02}, which states that 
$A^{\a}\subset W^+$ for $\a>\frac12$.
Summarizing, we obtain the following result.
\begin{theorem}  \label{thm-cantorreichmam}
For infinitely many values of $\xi$, the Cantor measure $\s(\xi)$ belongs to the Rajchman class $\mathcal{R}$,
and the corresponding Cantor set $E(\xi)=\supp\s(\xi)$ is the peak set for $W^+$. Consequently, there exist modulus support points and non-trivial 
modulus support functionals for the set $S_0$ from Section \ref{sect:Lomonosov-c0}.
\end{theorem}
\begin{remark}\label{manypeak}
Each peak function $g=g_E\in W^+$ generates a family of other peak functions in $W^+$, which correspond to the same peak set $E$. Precisely, let
$F$ be an analytic function in a neighborhood of $\ovl\D$. Since $g_E(\ovl{\D})\subset\ovl{\D}$, the Wiener--L\'evy theorem 
\cite[Theorem 6.2.16, (b)]{Simon-CCA4} states that the composition $G(z):=F(g_E(z))$ also lies in $W^+$. Let, in addition,
$$ |F(z)|<F(1)=1, \qquad z\in\D. $$
Then $G$ clearly satisfies \eqref{peak}, so $G$ is the peak function for $E$.

Here is an interesting particular case. Let $g_0=g_E(0)\not=0$, consider a Blaschke factor
$$ F(z):=e^{i\g}\,\frac{z-g_0}{1-\bar{g_0}z}\,, \qquad  e^{-i\g}=\frac{1-g_0}{1-\bar{g_0}}\,. $$
Then $G(z)=F(g_E(z))$ is the peak function for $W^+$, and $G(0)=0$.

In the above terminology, each modulus support functional generates an infinite family of other modulus support functionals with the same modulus support point.
\end{remark}

\begin{remark} \label{remark42}
It remains to establish the formal connection between the $c_0$-version of Lomonosov's example described in Section \ref{sect:Lomonosov-c0}, 
and the version from \cite{KLMW}. The latter version was written in the form in which the constant functions were already quotient out, like 
it was done in Remark \ref{remark21}. This means that, in order to get from our $S_0 \subset c_0$ to the version from \cite{KLMW}, one has to consider 
the subspace $E \subset c_0$ consisting of  vectors of the form $(\alpha, 0, 0, 0, \ldots)$ and to apply to $S_0$ the quotient map 
$q \mycolon c_0 \to c_0 / E$. The modified example is $q(S_0)$. Taking into account the natural identification of the quotient space 
$c_0 / E$ with the subspace $\tilde c_0 \subset c_0$ of vectors  $x = (x_j)_{j\ge0} \in c_0$  for which $x_0 = 0$, and  the natural identification 
of $(\tilde c_0)^*$ with the subspace $\tilde \ell_1 \subset \ell_1$ of vectors $a = (a_j)_{j\ge0} \in \ell_1$  for which $a_0 = 0$, one gets 
the representation of $q(S_0)$ from \cite{KLMW} in which the zero coordinates $x_0$ and $a_0$ are omitted, and the enumeration starts with the first coordinate. \\ 
\indent The important difference between $S_0$ and $q(S_0)$ is that in order to find a modulus support functional on  $q(S_0)$ one needs to 
build a modulus support functional  $a = (a_j)_{j\ge0} \in \ell_1$ on $S_0$ with the additional restriction $a_0 = 0$. In other words, 
one needs to find a peak function $G$ for $W^+$ with $G(0) = 0$. In Remark \ref{manypeak} we have demonstrated that this additional condition 
can be met, so the question from the introductory part of \cite{KLMW} about the existence of modulus attaining functionals on $q(S_0)$ solves in positive. 
The Problem 13.50 from \cite{KLMW} about the density in  $(\tilde c_0)^*$ of the set of modulus attaining functionals on $q(S_0)$ remains open.
\end{remark}


\begin{thebibliography}{99}

\bibitem{AlKal}   \textsc{F.~Albiac,  N.~Kalton},
 \emph{Topics in Banach space theory}.
Graduate Texts in Mathematics 233. Berlin: Springer. xi, 373~p. (2006). 

\bibitem{BP1}
  \textsc{E.~Bishop and R.~R.~Phelps},
  A proof that every Banach space is subreflexive,
  \emph{Bull. Amer. Math. Soc.} \textbf{67} (1961), 97--98.


\bibitem{BP2}
  \textsc{E.~Bishop and R.~R.~Phelps},
  The support functionals of a convex set,
  \emph{Proc. Sympos. Pure Math.} \textbf{7} (1963), 27--35.

\bibitem{HavJor-94}
 \textsc{V.~P.~Havin and B.~J\"oricke}
 \emph{The Uncertainty Principle in Harmonic Analysis}, Springer-Verlag, Berlin, 1994.
 
\bibitem{Hel10}
 \textsc{H.~Helson}
 \emph{Harmonic Analysis}, Hindustan Book Agency,  New Delhi, 2010.

\bibitem{Kad2018} 
 \textsc{V.~Kadets}, 
 \emph{A course in Functional Analysis and Measure Theory}. Translated from the Russian by Andrei Iacob. Universitext. Cham: Springer. xxii, 539~p. (2018).

\bibitem{KLMW}
  \textsc{V.~Kadets, G.~Lopez, M.~Mart\'{\i}n, and D.~Werner},
Norm attaining operators of finite rank. In:  Aron, Richard M. (ed.); Gallardo Guti\'errez, Eva A. (ed.); Martin, Miguel (ed.); Ryabogin, Dmitry (ed.); Spitkovsky, Ilya M. (ed.); Zvavitch, Artem (ed.). The mathematical legacy of Victor Lomonosov. Operator theory. Advances in Analysis and Geometry 2. Berlin: De Gruyter v, 300~p. (2020), 157--187. 

\bibitem{Koo-80}
 \textsc{P.~Koosis}
 \emph{Introduction to $H_p$ spaces}, CUP, Cambridge, 1980.


\bibitem{LinTza-01}
  \textsc{J.~Lindenstrauss and L.~Tzafriri},
  \emph{Classical Banach Spaces I: Sequence Spaces}, Springer-Verlag, Berlin, 1977.

\bibitem{Lom-2000}
  \textsc{V.~Lomonosov},
A counterexample to the Bishop-Phelps theorem in complex spaces,
  \emph{Israel J. Math.} \textbf{115} (2000),  25--28.

\bibitem{Lom-2000+}
  \textsc{V.~Lomonosov},
  On the Bishop-Phelps theorem in complex spaces,
  \emph{Quaest. Math.} \textbf{23} (2000), 187--191.

\bibitem{Lom-2001}
  \textsc{V.~Lomonosov},
  The Bishop-Phelps theorem fails for uniform non-selfadjoint dual operator algebras,
  \emph{J. Funct. Anal.} \textbf{185} (2001), 214--219.
  
\bibitem{NoWol-89}
 \textsc{A.~Noell and T.~Wolff},
 On peak sets for Lip $\alpha$ classes,
 \emph{J. Funct. Anal.}  \textbf{86} (1989), 136--179.
 
\bibitem{No}
 \textsc{A.~Noell},
 Peak sets and boundary interpolation sets for the unit disc: a survey, arXiv:1905.07441.  

\bibitem{Phelps1991}
  \textsc{R.~R.~Phelps},
  The Bishop-Phelps theorem in complex spaces: an open problem,
  \emph{Lecture Notes in Pure and Applied Mathematics} \textbf{136} (1991), 337--340.

\bibitem{Simon-CCA4}
  \textsc{B.~Simon},
  \emph{A Comprehensive Course in Analysis. Part 4:Operator Theory}, AMS, Providence, RI, 2015.

\bibitem{Zyg-02}
 \textsc{A.~Zygmund}
 \emph{Trigonometric Series}, 3d ed., Cambridge Math. Library, CUP, Cambridge, 2002.

\end{thebibliography}
\end{document}